\documentclass[a4paper,abstract=on, enabledeprecatedfontcommands]{scrartcl}

\usepackage{geometry}
\usepackage{datetime2}
\DTMsetstyle{iso}
\usepackage{amsmath, amssymb, amsthm}
\usepackage{cases}
\usepackage{graphicx}
\usepackage{color}
\usepackage{hyperref}
\usepackage{enumitem}
\usepackage{ascmac}
\usepackage[T1]{fontenc}
\usepackage{lmodern}
\usepackage[T1,euler-digits]{eulervm}

\numberwithin{equation}{section}

\newtheorem{theorem}{Theorem}[section]
\newtheorem{proposition}[theorem]{Proposition}

\theoremstyle{definition}
\newtheorem{definition}[theorem]{Definition}
\theoremstyle{remark}
\newtheorem{remark}[theorem]{Remark}
\newtheorem{lemma}[theorem]{Lemma}

\title{\textsf{Graph reliability evaluation via random $K$-out-of-$N$ systems}}

\author{
  Hiroaki Mohri%
  \thanks{
    Faculty of Commerce, Waseda University,
    Shinjuku-ku, Tokyo, Japan}
\and
  Jun-ichi Takeshita%
  \thanks{Research Institute of Science for Safety and Sustainability,
	National Institute of Advanced Industrial Science and Technology (AIST),
  Tsukuba,
  Japan. (\texttt{jun-takeshita@aist.go.jp})}
}

\date{\normalsize\DTMnow}

\begin{document}
\maketitle

\begin{abstract}
  The present study was concerned with network failure problems for simple connected undirected graphs.
  A connected graph becomes unconnected through edge failure, under the assumptions that only edges can fail and each edge has an identical failure distribution.
  The main purpose of the present study was to show recurrent relations with respect to the number of edges in graph generation procedures.
  To this end, simple connected undirected graphs were correlated to random $K$-out-of-$N$ systems, and key features of such systems were applied. In addition, some simple graph cases and examples were analyzed.

  \noindent
\textbf{keywords}: Simple connected undirected graphs, network failure problems, random $K$-out-of-$N$ systems, recurrent relations
\end{abstract}

\section{Introduction}

The present study was concerned with simple connected undirected graphs of the type, $G=(V(G),E(G))$,
where each element in $V(G)$ is called a vertex, and each element in $E(G) \subset V(G)^2$ is called an edge,
which comprises a pair of vertices.
A graph is said to be undirected if its edges do not have a direction.
A graph is said to be simple if there are not more than two edges between the same pair of vertices, and connected if there is at least one path between any pair of vertices.
As graphs are useful for illustrating network structures, we sometimes use them to model infrastructures such as electrical, transportation, and water networks in and among cities. One of the primary analytical interests regarding such infrastructures is reliability evaluation.

There are two main fields concerned with graph reliability: reliability engineering and discrete system study.
The former assesses the reliability of redundant systems with two special vertices (a sink and a source), under the assumption that each vertex and each edge has a failure distribution function that depends on time;
and typically analyzes the failure distribution function for whole systems, mean time to failure, optimal replacement policy, etc. (see Barlow and Proschan~\cite{Barlow.Proschan_1965__MathematicalTheory} and Nakagawa~\cite{Nakagawa_2008__AdvancedReliability}, for fundamental texts in this field).
The latter assesses reliability in terms of the probability of connectedness in graphs with a constant probability of edge failure (e.g., Colbourn\cite{Colbourn_1987__CombinatoricsNetwork}), seeking to properly estimate or precisely calculate the reliability of very large graphs, based on the connectivity among the vertices in the graphs. 

In modeling infrastructural networks and evaluating their reliability,
we should consider connectivity among all the vertices,
under the assumption that each vertex and edge has a failure function that depends on time.
However, the studies cited above do not include such reliability analysis.

The present study, then, focused on simple undirected connected graphs,
and assumed that only edges can fail in the graphs,
based on a failure distribution function that depends on time.
When a connected graph becomes unconnected through edge failure,
we call it ``Network failure.''
In order to model network failure problems,
we used the idea of random $K$-out-of-$N$ systems,
originally proposed by Ito and Nakagawa~\cite{Ito.Nakagawa_2019_Systemsengineering:reliabilityanalysisusingk-out-of-nstructures_ReliabilityProperties}.
The system operates when at least $K$ units out of a total of $N$ units operate,
and $K$ is a random variable.
When each edge is considered to correspond to a unit in a random $K$-out-of-$N$ system,
we can analyze the graph reliability by applying the knowledge of random $K$-out-of-$N$ systems.
To determine the probability distribution $P(K \leq k)$ for a given graph,
we show the recurrent relations with respect to the number of edges in the graph generation procedure.

The present paper is organized as follows:
Section 2 summarizes basic notions of graphs and $K$-out-of-$N$ systems,
as well as preliminary facts.
Section 3 defines graph reliability,
and Section 4 summarizes the main results of the study and provides relevant proofs.
Section 5 describes the results of simple graph-structure cases,
and Section 6 provides an example of graph generation.

\section{Preliminaries}

\subsection{Basic graph definitions}

Let $G=(V(G), E(G))$ be a graph, where $V(G)$ and $E(G)$ are the sets of vertices and edges of $G$, respectively.
First, we prepare some common terms in graph theory.
There are several good texts on graph theory,
but we mainly refer to two works:
Bondy and Murty~\cite{Bondy.Murty_2008__GraphTheory}, \cite{Bondy.Murty_1976__GraphTheory}.

A vertex is \emph{incident} to an edge if the vertex is one of the two vertices connected by the edge.
The two incident vertices are said to be \emph{adjacent} and are called \emph{neighbors}.
An edge is a \emph{loop} if the two incident vertices are identical;
while an edge is a \emph{link} if the two incident vertices are distinct.
Two or more links are \emph{parallel edges} if the links have the same two incident vertices.

\begin{definition}[walk, path, closed walk]\mbox{}
  \begin{itemize}[nosep]
  \item A \emph{walk} in $G$ is a sequence of alternating vertices and edges,
    $\pi = v_0 e_1 v_1 e_2 \ldots e_k v_k$,
    such that the two incident vertices of $e_i \ (i=1, \ldots, k)$ are $v_{i-1}$ and $v_i$;
    specifically, the walk $\pi$ is called $v_0 v_k$-walk.
  \item A walk $\pi$ is called a \emph{path} if the vertices $v_i \ (i=0, \ldots, k)$ and the edges $e_j \ (j=1, \dots, k)$ are distinct;
    specifically, the path $\pi$ is called $v_0 v_k$-path.
  \item A walk $\pi$ is said to be \emph{closed} if the first vertex $v_0$ and the final vertex $v_k$ are identical.
  \end{itemize}  
\end{definition}

In the present study, we consider only simple connected undirected graphs,
which are defined as follows.
\begin{definition}[undirected, simple, connected]\mbox{}
  \begin{itemize}[nosep]
  \item An \emph{undirected graph} is a graph in which all the edges are bidirectional. 
  \item A graph is said to be \emph{simple} if the graph has neither loops nor parallel edges;
  \item A graph is said to be \emph{connected} if, for any two distinct vertices $v_1, v_2 \in V(G)$, there is a walk between them.
  \end{itemize}
\end{definition}


\subsection{Random $K$-out-of-$N$ systems}

In this subsection, we summarize key features of conventional $K$-out-of-$N$ systems (Barlow and Proschan~\cite{Barlow.Proschan_1965__MathematicalTheory}) and random $K$-out-of-$N$ systems.
The random type was originally introduced by Ito and Nakagawa~\cite{Ito.Nakagawa_2019_Systemsengineering:reliabilityanalysisusingk-out-of-nstructures_ReliabilityProperties}.

\begin{definition}[Conventional and random $K$-out-of-$N$ system]\mbox{}
  Both conventional and random $K$-out-of-$N$ systems cannot work if and only if at least $K$ units of the total $N$ units are not operaional.
  \begin{itemize}[nosep]
  \item the system is called a \emph{conventional} $K$-out-of-$N$ system if both $K$ and $N$ are constant;
  \item the system is called a \emph{random} $K$-out-of-$N$ system if $N$ is constant but $K$ is a random variable.
  \end{itemize}
\end{definition}

For random $K$-out-of-$N$ systems,
Ito and Nakagawa~\cite{Ito.Nakagawa_2019_Systemsengineering:reliabilityanalysisusingk-out-of-nstructures_ReliabilityProperties} derived the failure function of the system at time $t$ as follows.
\begin{lemma}\label{lem:rkofn_failfunc}
  Suppose that, for a given constant integer $N$,
  $K$ is a random variable with the probability function $p_k := P(K = k) \ (k=1, \ldots, N)$,
  and each unit has the identical failure distribution function $F(t) \ (t \geq 0)$.
  Then the failure distribution function of the system at time $t$, say $\mathcal{F}(t)$, is given by
  \begin{equation}
    \label{eq:1}
    \mathcal{F}(t)
    = \sum_{k=1}^N P_k \binom{N}{k} \left(1 - F(t)\right)^k F(t)^{N-k}.
  \end{equation}
\end{lemma}

\begin{remark} Ito and Nakagawa~\cite{Ito.Nakagawa_2019_Systemsengineering:reliabilityanalysisusingk-out-of-nstructures_ReliabilityProperties} originally derived the reliability of the system at time $t$.
  Thus, the original $P_k$ corresponds to $1-P_k$ in the present study.
\end{remark}

\section{Formulation of the graph reliability}

In the present study, we assume that only the edges of simple connected undirected graphs can fail,
and we define graph failure as follows.
\begin{definition}[Failure of graphs]
  A simple connected undirected graph $G=(V(G),E(G))$ is said to be \emph{in fail} if it is not connected.
\end{definition}

The number of failed edeges that result in graph failure depends on the structure of the graph.
In other words, which and how many edges fail determines whether the graph is in fail or not.
For a given simple connected undirected graph $G=(V(G),E(G))$ with $|E(G)|=M$,
we consider each edge of the graph to correspond to a unit in a random $K$-out-of-$M$ system.
Also, we assume that each edge has an identical independent failure distribution function $F(t)$.
Then the distribution function of $K$, that is $P_k := P(K \leq k) \ (k=1, \ldots, M)$,
corresponds to the probability that the graph becomes in fail when choosing $k$ out of $M$ edges.
Thus, we can determine the failure function of a graph by determining the distribution function of $K$ in the corresponding random $K$-out-of-$M$ system.


\section{Main theorems}

First, in order to describe the main results, we introduce some terms.
\begin{definition}[$G$-disconnected set, $G$-cut set]
  Let $G=(V(G), E(G))$ be a simple connected undirected graph.
  Then
  \begin{itemize}[nosep]
  \item an edge set $F \subset E(G)$ is called a \emph{$G$-disconnected set} if the graph $G-F$ is not connected;
  \item an edge set $F \subset E(G)$ is called a \emph{$G$-cut set} if $F$ is a $G$-disconnected set,
    but for any proper subset $F' \subsetneq F$, $F'$ is not a $G$-disconnected set.
  \end{itemize}
  Further, if $|F| = k$, we call its size $k$.
\end{definition}

\begin{definition}[$uv$-disconnected set, $uv$-cut set]
  Let $u,v \in V(G)$ be two distinct vertices of a graph $G=(V(G), E(G))$, and $F \subset E(G)$.
  Then
  \begin{itemize}[nosep]
  \item an edge set $F \subset E(G)$ is called a \emph{$uv$-disconnected set} of graph $G$ if no paths from $u$ to $v$ exist in the graph $G-F$;
  \item an edge set $F \subset E(G)$ is called  a \emph{$uv$-cut set} of graph $G$ if $F$ is a $uv$-disconnected set,
    but for any proper subset $F' \subsetneq F$, $F'$ is not a $uv$-disconnected set.
  \end{itemize}
  Further, if $|F| = k$, we call its size $k$.

  See Figure~\ref{fig:cutset} for examples of $uv$-disconnected and $uv$-cut sets.
\begin{figure}[htbp]
  \centering
  \includegraphics[page=3,scale=.5]{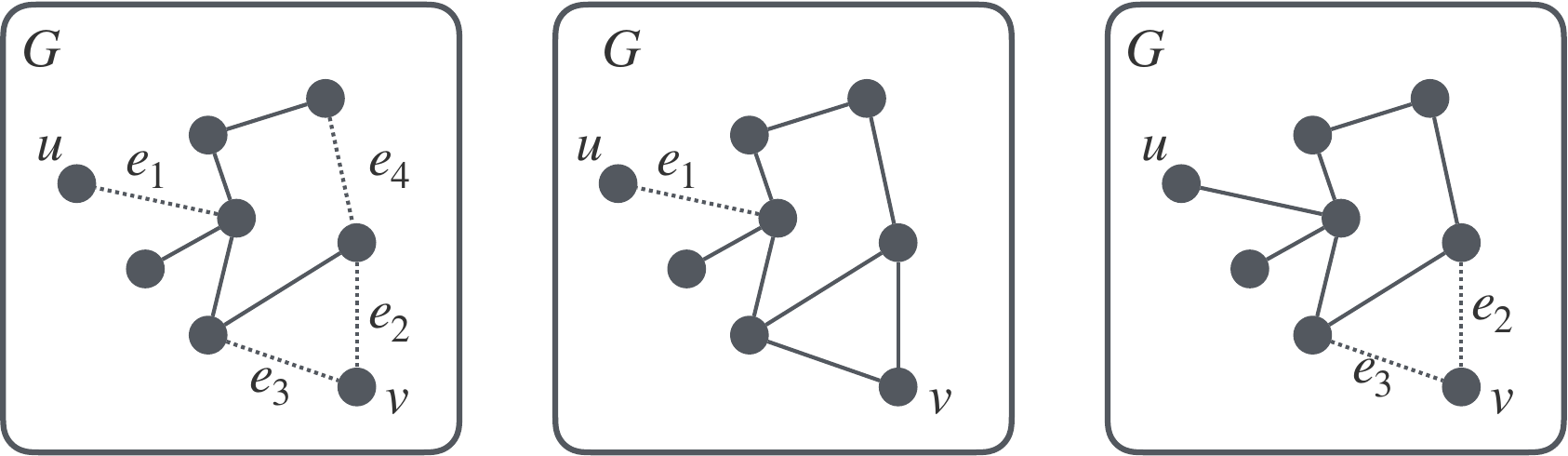}
  \caption{The edge set $\{e_1, e_2, e_3, e_4\}$ is a $uv$-disconnected set, but not a $uv$-cut set.
    However, the set includes two $uv$-cut sets, $\{e_1\}$ and $\{e_2, e_3\}$.  \label{fig:cutset}}
\end{figure}
\end{definition}

Using the terms above,
the failure of a graph $G$,
when $k$ edges have failed,
is equivalent to choosing a $G$-disconnected set of the graph $G$ with size $k$.
Hence, the distribution function $P_k$ of the corresponding random $K$-out-of-$M$ system is the ratio of $G$-disconnected sets with size $k$ to all the edge sets with size $k$.

Let $H=(V(H),E(H))$ be a simple connected undirected graph with $N$ vertices and $M-1$ edges
(that is, $|V(H)|=N$ and $|E(H)|=M-1$);
and $P_k^H := P(K \leq k)$ is the distribution function of $K$ in the corresponding random $K$-out-of-$(M-1)$ system.
We consider a simple connected undirected graph $G=(V(G),E(G))$, constructed by adding an edge to the graph $H$.
Then $|V(G)|$ is either $N$ or $N+1$, and $|E(G)|=M$.
In fact, if an end of the additional edge is a vertex $v \in V(H)$ and the other end is a new vertex,
then $|V(G)|=N+1$ and $|E(G)|=M$;
while, if the ends of the additional edges are two discontiguous vertices $u, v \in V(H) \ (u \neq v)$,
then $|V(G)|=N$ and $|E(G)|=M$.

For the case $|V(G)|=N+1$, the following holds.
\begin{theorem}\label{thm:N+1}
  Let $M \geq 2$ and $N \geq 1$ be integers,
  $H=(V(H), E(H))$ be a simple connected undirected graph with $|V(H)|=N$ and $|E(H)| = M-1$,
  and the distribution function of the corresponding random $K$-out-of-$(M-1)$ system be $P^H_
  k \ (k=1, \ldots, M-1)$.
  We take $v_0 \in V(H)$ and a new vertex $v^+ \not\in V(H)$,
  and construct a graph $G=(V(G), E(G))$ with $V(G):=V(H) \cup \{v^+\}$ and $E(G):=E(H) \cup \{(v_0 v^+)\}$.
  Also, let $P^G_k$ be the distribution function of the $K$-out-of-$M$ system corresponding to the graph $G$.
  Then $P^G_k$ can be described by using $P^H_k$, as follows:
  \begin{equation}
    \label{eq:N+1}
      P^G_k =
      \begin{cases}
        \displaystyle \left[ P^H_k \binom{M-1}{k} + \binom{M-1}{k-1} \right] \Big/ \binom{M}{k} & \text{if} \quad  1 \leq k \leq M-1 \text{ and } P^G_{k-1} < 1,\\[10pt]
        1 & \text{otherwise}.
      \end{cases}
    \end{equation}
    We here define $\binom{M-1}{0} = 1$ for any $M \geq 2$.
  \end{theorem}

  \begin{proof}
  Consider a graph $G=(V(G), E(G))$, constructed by adding one vertex and one edge, say $e^+$.
  Then we can divide when the graph $G$'s failure as a result of $j \geq 1$ failing edges into two cases:
  \begin{enumerate}[label=(\roman*), nosep]
  \item The new edge fails;
  \item The new edge does not fail, but part of the graph $H$ fails.
  \end{enumerate}
  The corresponding $G$-disconnected sets with size $k$ in the two cases above are, respectively, $\{e^+\} \cup F \ (F \subset E(H), |F|=k-1)$ and $H$-disconnected sets with size $k$.
  Note that $F = \emptyset$ when $k=1$.
  The number of the sets $\{e^+\} \cup F \ (F \subset E(H), |F|=k-1)$ is $\binom{M-1}{k-1}$;
  while that of the $H$-connected sets with size $k$ is $P^H_k \binom{M-1}{k}$.

  Furthermore, it is clear that $P^G_{k} = 1$ if $P^G_{k-1} = 1$.
  Therefore, we have equation~\eqref{eq:N+1}. 
\end{proof}

For the case $|V(G)|=N$, the following holds.
\begin{theorem}\label{thm:N}
  Let $M \geq 2$ and $N \geq 1$ be integers,
  $H=(V(H), E(H))$ be a simple connected undirected graph with $|V(H)|=N$ and $|E(H)| = M-1$,
  and the distribution function of its corresponding random $K$-out-of-$(M-1)$ system be $P^H_
  k \ (k=1, \ldots, M-1)$.
  We take two discontiguous vertices $u, v \in V(H) \ (u \neq v)$ and construct a graph $G=(V(G), E(G))$ with $V(G):=V(H)$ and $E(G):=E(H) \cup \{(uv)\}$.
  Also, let $P^G_k$ be the distribution function of the $K$-out-of-$M$ system corresponding to the graph $G$.
  Then $P^G_k$ can be described by using $P^H_k$, as follows:
  \begin{equation}
    \label{eq:N}
    P^G_k =
    \begin{cases}
      \displaystyle \left[ P^H_k \binom{M-1}{k} + P^H_{k-1} \binom{M-1}{k-1} - n^H_k \right] \Big/ \binom{M}{k} & \text{if} \quad  1 \leq k \leq M-1 \text{ and } P^G_{k-1} < 1,\\[10pt]
      1 & \text{otherwise},
    \end{cases}
  \end{equation}
  where $n^H_k$ stands for the number of $H$-disconnected sets with size $k$, say $F^H_k$, which includes only one $uv$-cut set, say $F'$, and $F^H_k \setminus F'$ is not a $H$-disconnected set.
Also, we here define $P^H_0 \binom{M-1}{0} = 0$ for any $M \geq 2$.
\end{theorem}

\begin{proof}
  Consider two discontiguous vertices $u, v \in V(H) \ (u \neq v)$ and a graph $G=(V(G), E(G))$ with $V(G) = V(H)$ and $E(G) = E(H) \cup \{(uv)\}$. 
Then we can divide the $G$-disconnected sets with size $k$, say $F^G_k$, into two cases:
\begin{enumerate}[label=(\roman*), nosep]
\item $(uv) \in F^G_k$;
\item $(uv) \not \in F^G_k$.
\end{enumerate}

Since Case (i) means the edge $(uv)$ fails, $F^G_k$ for Case (i) is a $G$-disconnected set if and only if $F_G \setminus \{(uv)\}$ is an $H$-disconnected set with size $k-1$.
In other words, the number of $H$-disconnected sets with size $k-1$, that is $P^H_{k-1} \binom{M-1}{k-1}$, is equal to that of $F^G_k$ for Case (i).

In order to count the number of $F^G_k$ in Case (ii),
we consider the $H$-disconnected sets with size $k$, say $F^H_k$.
For any $F^H_k$, there exists the set of $H$-cut sets $\mathcal{F}_H := \{F_i \subset E(H)\, | \, i=1, \ldots, l \text{ and } F_i \cap F_j = \emptyset \ (i \neq j)\}$, and a non-$H$-disconnected set $F' \subset E(H)$, such that
\begin{equation}
  \label{eq:divFH}
  F^H_k = \left(\bigcup_{i=1}^l F_i  \right) \cup F'.
\end{equation}
From~\eqref{eq:divFH}, we can divide Case (ii) into three cases:
\begin{enumerate}[label=(\alph*), nosep]
\item $l \geq 2$;
\item $l=1$ and $F_1$ is a $uv$-cut set;
\item $l=1$ and $F_1$ is not a $uv$-cut set.
\end{enumerate}

For Case (a), since the graph $H - F^H_k$ has more than three connected components,
the sets $F^H_k$ are still $G$-disconnected sets with size $k$.

Also, for Case (c), the sets $F^H_k$ are still $G$-disconnected sets.
Since $F_1$ is not a $(uv)$-cut set in $H$, there exist two distinct vertices $w_1, w_2 \in V(H) \ (w_1 \neq w_2)$,  such that $F_1$ is a $(w_1 w_2)$-cut set and at least one of $w_1$ and $w_2$ can reach neither $u$ nor $v$ in $H - F^H_k$.
If both $w_1$ and $w_2$ can reach either $u$ and $v$, then $w_1$ and $w_2$ are connected in $H - F^H_j$.
Indeed, since there exists a path between $u$ and $v$, one can reach from $w_1$ to $w_2$ through the $uv$-path.
This implies that there is no $w_1 w_2$-walk in $G$ and $F^H_k$ are $G$-disconnected sets with size $k$.

On the other hand, for Case (b), there are no $uv$-walks, but all the vertices can reach either $u$ or $v$. 
If there exists a vertex $w$ that can reach neither $u$ nor $v$, the three distinct vertices $u$, $v$, and $w$ cannot reach each other in $H - F^H_k$.
This implies that $H - F^H_k$ has at least three connected components, which contradicts $l=1$.
Therefore, by adding the edge $(uv)$ to $H - F^H_k$, all the vertices can reach both $u$ and $v$ in $H - F^H_k + (uv)$.
Thus, only in Case (b), the $H$-disconnected sets with size $k$ are not $G$-disconnected sets with size $k$.
Here we denote the number of $H$-disconnected sets belonging to Case (b) by $n^H_k$.

Since the number of $H$-disconnected sets with size $k$ is $P^H_k \binom{M-1}{k}$,
we have $P^H_k \binom{M-1}{k} - n^H_k$, which is the number of $G$-disconnected sets with size $k$ in Case (ii).

Furthermore, it is clear that $P^G_k=1$ if $P^G_{k-1}=1$.
Thus, we have equation~\eqref{eq:N}. 
\end{proof}

From these theorems, we have the following main theorem.
\begin{theorem}
  Let $M \geq 1$ and $N \geq 1$ be integers,
  $G_M =(V(G_M), E(G_M))$ be a simple connected undirected graph with $|V(G_M)|=N$ and $|E(G_M)|=M$,
  and $P^{M}_{k}$ be the distribution function of its corresponding random $K$-out-of-$M$ system.
  Then the recurrence relations of $P^{M}_{k}$ with respect to the number of edges $M$ are as follows:
  \begin{itemize}
  \item When $M=1$,
  \begin{equation}
    P^1_k = 1 \ \text{for any $k$}\label{main:eq10};
   \end{equation}
 \item When $M \geq 2$,
      \begin{numcases}{P^{M}_{k} = }
        \displaystyle \left[ P^{M-1}_{k} \binom{M-1}{k} + P^{M-1}_{k-1} \binom{M-1}{k-1} - n^{M-1}_{k} \right] \Big/ \binom{M}{k} & \label{main:eq21}\\[10pt]
          \qquad\text{if} \quad V(G_M) = V(G_{M-1}), E(G_M)=E(G_{M-1}) \cup \{(u,v)\}, \notag\\
          \qquad\quad\text{ where } u, v \in V(H) \text{ and discontiguous};
          1 \leq j \leq M-1; \text{ and } P^G_{k-1} < 1, \notag\\[10pt]
          \displaystyle \left[ P^{M-1}_k \binom{M-1}{k} + \binom{M-1}{k-1} \right] \Big/ \binom{M}{k} &
          \label{main:eq22}\\[10pt]
          \qquad\text{if} \quad V(G_M) = V(G_{M-1}) \cup \{v^+\}, E(G_M)=E(G_{M-1}) \cup \{(u,v^+)\},
          \notag\\
          \qquad\quad\text{ where } u \in V(G_{M-1}) \text{ and } v^+ \not\in V(G_{M-1});
          \ 1 \leq k \leq M-1; \text{ and } P^G_{k-1} < 1,\notag \\[10pt]
          1 \quad  \text{otherwise}, \label{main:eq23}&
        \end{numcases}
  where we define $\binom{M-1}{0} = 1$ and $P^{M-1}_0 \binom{M-1}{0} = 0$ for any $M \geq 2$,
  and $n^{M-1}_{k}$ is the number of $(uv)$-disconnected sets of $G_{M-1}$ that include only one $(uv)$-cut set.
\end{itemize}
\end{theorem}

\section{Simple cases}

This section deals with two simple cases: a tree and a cycle (Figure~\ref{fig:tree_cycle}).
In these cases, we do not have to use the recurrent relations to determine their corresponding random $K$-out-of-$N$ systems.

\begin{definition}[tree, cycle]\mbox{}
  \begin{itemize}[nosep]
  \item A graph $G$ is called a \emph{tree} if the graph has no closed walks and is connected.
  \item A graph $G$ is called a \emph{cycle} if the graph is a closed path.
  \end{itemize}
\end{definition}
\begin{figure}[htbp]
  \centering
  \includegraphics[page=2,scale=.5]{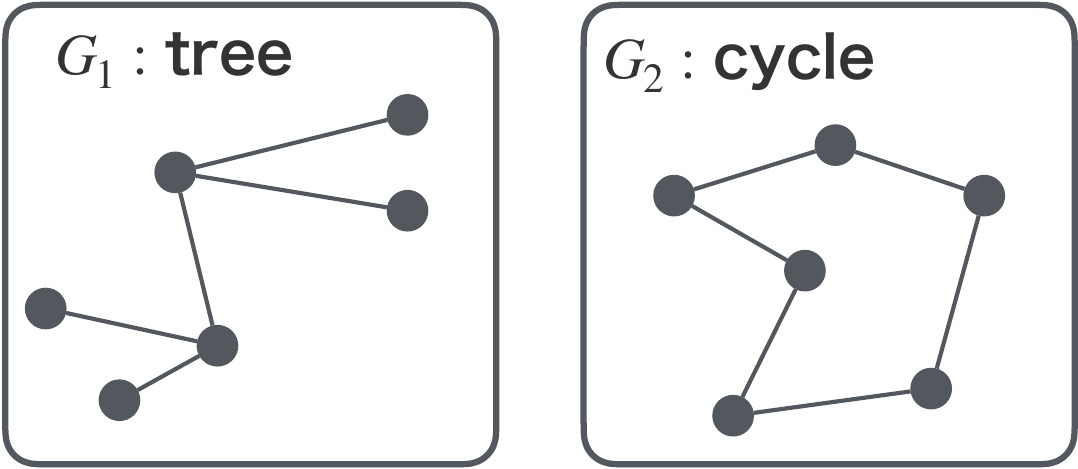}
  \caption{Graph $G_1$ is an example of a tree, graph $G_2$ that of a cycle.}
  \label{fig:tree_cycle}
\end{figure}

First, we describe well-known properties of trees and cycles.

\begin{lemma}[Properties of trees]\label{lem:tree}
  Let $G=(V(G), E(G))$ be a simple undirected graph with $|V(G)| = N$.
  If $N \geq 2$, then
  the following (a)-(d) are equivalent:
  \begin{enumerate}[nosep, label=(\alph*)]
  \item $G$ is connected and has no closed walks, which is the definition of a tree;
  \item $G$ is connected and has $N-1$ edges;
  \item $G$ is connected, but if one edge is deleted then $G$ is not connected;
  \item there is only one path from a given vertex to another distinct vertex.
  \end{enumerate}
\end{lemma}

\begin{lemma}[Properties of cycles]\label{lem:cycle}
  A closed walk $\pi$ is a cycle if and only if one cannot make another closed walk by deleting an edge from $\pi$.
\end{lemma}

For trees, the following propositions immediately hold from Lemma~\ref{lem:tree}.
\begin{proposition}
  If $G=(V(G), E(G))$ is a tree with $|E(G)|= M$,
  then $G$ is a random $K$-out-of-$M$ system with the distribution functions $P^G_i = 1$ for any $i \geq 1$.
\end{proposition}

\begin{proposition}
  If $G=(V(G), E(G))$ is a simple connected undirected graph with $|V(G)|=N$ and $|E(G)|=N-1$,
  then $G$ is a random $K$-out-of-$(N-1)$ system with the distribution function $P^G_i = 1$ for any $i \geq 1$.
\end{proposition}

Regarding cycles, the following propositions immediately hold from Lemma~\ref{lem:cycle}.
\begin{proposition}
  If $G=(V(G), E(G))$ is a cycle with $|E(G)|=M$,
  then $G$ is a random $K$-out-of-$M$ system with the distribution functions $P^G_1 = 0$ and $P^G_i =1$ for any $i \geq 2$.
\end{proposition}

\begin{proposition}
  If $G=(V(G), E(G))$ is a cycle with $|V(G)|=N$,
  then $G$ is a random $K$-out-of-$N$ system with the distribution functions $P^G_1 = 0$ and $P^G_i = 1$ for any $i \geq 2$.
 \end{proposition}

 \section{Examples}

This section illustrates graph generation, from a graph with one edge to a graph with five edges (Figure~\ref{fig:ex}).
 
 Let $P^i_k$ be the cumulative probability up to $i$ of $G_i$.
 Then $P^1_k = 1$ for any $k$ from \eqref{main:eq10}.
 From \eqref{main:eq22} and \eqref{main:eq23},
 \begin{equation*}
   P^2_1 = \left[P^1_1 \binom{1}{1} + 1\right] \big/ \binom{1}{2} = 1, \quad P^2_k = 1 \ (k \geq 2).
 \end{equation*}
 From \eqref{main:eq21} and \eqref{main:eq23},
 \begin{align*}
   P^3_1 &= \left[P^2_1 \binom{2}{1} + P^2_0 \binom{2}{0} - 2\right] \big/ \binom{1}{3} = 0,\\
   P^3_2 &= \left[P^2_2 \binom{2}{2} + P^2_1 \binom{2}{1} - 0\right] \big/ \binom{2}{3} = 1,\\
   P^3_k &= 1 \ (k \geq 3).
 \end{align*}
 From \eqref{main:eq22} and \eqref{main:eq23},
 \begin{align*}
   P^4_1 &= \left[P^3_1 \binom{3}{1} + \binom{3}{0}\right] \big/ \binom{1}{4} = \frac{1}{4},\\
   P^4_2 &= \left[P^3_2 \binom{3}{2} + \binom{3}{1}\right] \big/ \binom{2}{4} = 1,\\
   P^4_k &= 1 \ (k \geq 3).
 \end{align*}
 From \eqref{main:eq21} and \eqref{main:eq23},
 \begin{align*}
   P^5_1 &= \left[P^4_1 \binom{4}{1} + P^4_0 \binom{4}{0} - 1\right] \big/ \binom{1}{5} = 0\\
   P^5_2 &= \left[P^4_2 \binom{4}{2} + P^4_1 \binom{4}{1} - 5\right] \big/ \binom{2}{5} = \frac{1}{5},\\
   P^5_3 &= \left[P^3_1 \binom{4}{1} + P^4_2 \binom{4}{2} - 0\right] \big/ \binom{3}{5} = 1,\\
   P^5_k &= 1 \ (k \geq 3).
 \end{align*}

 \begin{figure}[htbp]
   \centering
   \includegraphics[page=4,scale=.5]{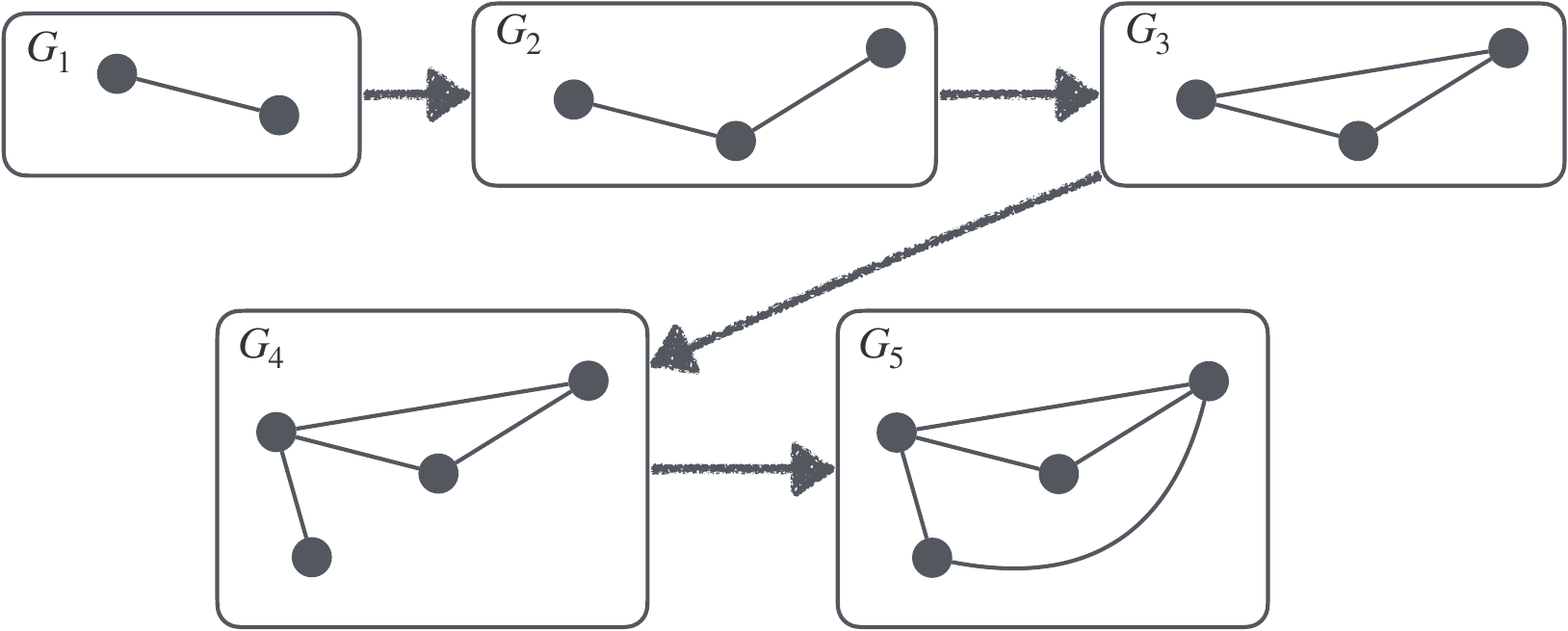}
   \caption{Graph generation, from a graph with one edge ($G_1$) to a graph with five edges ($G_5$).}
   \label{fig:ex}
 \end{figure}

 \bibliographystyle{siam}
 \bibliography{GraphKoutofN.bib}
 
\end{document}